\documentclass{siamltex1213}

\usepackage{amsmath,amssymb,amsfonts,cmmib57,mathdots,xcolor,mathrsfs}
\usepackage[latin1]{inputenc}
\usepackage{enumerate}
\usepackage[linesnumbered,lined,commentsnumbered]{algorithm2e}

\newtheorem{remark}[theorem]{Remark}

\title{Mixed forward--backward stability of the two--level orthogonal Arnoldi method for quadratic problems}
\author{Karl Meerbergen\footnotemark[1] and Javier P\'{e}rez\footnotemark[2]}
\begin{document}
\maketitle
\slugger{simax}{xxxx}{xx}{x}{x--x}
\renewcommand{\thefootnote}{\fnsymbol{footnote}}
\footnotetext[1]{Department of Computer Science, KU Leuven, Celestijnenlaan 200A, 3001 Heverlee, Belgium. 
Email {\tt karl.meerbergen@cs.kuleuven.be}.}
\footnotetext[2]{Department of Computer Science, KU Leuven, Celestijnenlaan 200A, 3001 Heverlee, Belgium. Email: {\tt javier.perezalvaro@kuleuven.be}. 
Supported by KU Leuven Research Council grant OT/14/074 and the Interuniversity Attraction Pole DYSCO, initiated by the Belgian State Science Policy Office.}

\renewcommand{\thefootnote}{\arabic{footnote}}

\begin{abstract}
We revisit the numerical stability of the two--level orthogonal Arnoldi (TOAR) method for computing an orthonormal basis of a second--order Krylov subspace associated with two given matrices.
We show that the computed basis is close (on certain subspace metric sense) to a basis for a second--order Krylov subspace associated with nearby coefficient matrices,  provided that the norms of the given matrices are not too large or too small.
Thus, the results in this work provide for the first time conditions that guarantee the numerical stability of the TOAR method in computing  orthonormal bases of second--order Krylov subspaces.
We also study scaling the quadratic problem for improving the numerical stability of the TOAR procedure  when the norms of the matrices are too large or too small.
We show that in many cases the TOAR procedure applied to   scaled matrices is numerically stable when the scaling introduced by Fan, Lin and Van Dooren is used. 
\end{abstract}

\pagestyle{myheadings}
\thispagestyle{plain}

\section{Introduction}\label{sec:intro}

Given two complex matrices $A,B\in\mathbb{C}^{n\times n}$ and two starting vectors $r_{-1},r_0\in\mathbb{C}^{n}$, if we define the sequence $r_{-1},r_0,r_1,\hdots,r_{k-1}$ by the recurrence relation
\[
r_i = Ar_{i-1}+Br_{i-2}, \quad \mbox{for }i=1,2,\hdots,k-1,
\] 
then, the \emph{second--order Krylov subspace} associated with $A$ and $B$, introduced by Bai and Su \cite{SOAR}, is the subspace
\begin{equation*}\label{eq:second-order-Krylov}
\mathcal{G}_k(A,B;r_{-1},r_0):=\mathrm{span}\{r_{-1},r_0,r_1,\hdots,r_{k-1}\}.
\end{equation*}

Projection methods based on second--order Krylov subspaces have been found to be reliable procedures for obtaining good approximations to the solutions of (structured) quadratic eigenvalue problems \cite{SOAR,TOAR}, and for model order reduction of second--order dynamical systems \cite{Bai2005,TOAR} and second--order  time--delay systems \cite{delay}.
These procedures start by computing an orthonormal set of vectors  $\{q_1,q_2,\hdots,q_{k+1}\}$ such that 
\[
\mathrm{span}\{q_1,q_2,\hdots,q_{k+1}\}=\mathcal{G}_k(A,B;r_{-1},r_0).
\]
They continue by projecting the problem onto the subspace $\mathcal{G}_k(A,B;r_{-1},r_0)$, reducing the size of the original problem.
Finally, the projected problem is solved by using standard  algorithms for small/medium--sized dense matrices.
The convergence of these projection methods for quadratic eigenvalue problems is studied in \cite{q-Ritz}.

The \emph{second--order Arnoldi (SOAR)} method   and the \emph{two--level orthogonal Arnoldi (TOAR)} method \cite{SOAR,TOAR,TOAR_origin}, are two well--known algorithms for computing orthonormal bases of  second--order Krylov subspaces.
Both methods compute such bases by embedding the second--order Krylov subspaces into standard Krylov subspaces.
Moreover, while the SOAR method is prone to numerical instability \cite{TOAR}, the analysis performed in \cite{TOAR} provides solid theoretical evidence of the numerical stability of the TOAR method.
More precisely, the TOAR method is backward stable in computing an orthonormal basis of the Krylov subspace in which the second--order Krylov subspace is embedded.
In this work, we extend this result by showing that the computed orthonormal basis for the second--order Krylov subspace is close (in the standard subspace metric sense \cite{Pete}) to a second--order Krylov subspace associated with matrices $A+\Delta A$ and $B+\Delta B$.
This result is stated in Corollary \ref{cor:stability}, which is a consequence of the more general Theorem \ref{thm:main}.
These two results are the major contributions of this work.
Additionally, in Section \ref{sec:scaling}, we study how scaling the original quadratic problem affects the norms of $\Delta A$ and $\Delta B$. 

The notation used in the rest of the paper is as follows.
We use lowercase letters for vectors  and uppercase letters for matrices.
In addition, we use boldface letters to indicate that a matrix (resp. vector) will be considered as a $2\times 1$ block--matrix (resp. block--vector), whose blocks are denoted with superscripts as in
\[
\mathbf{A} =
\begin{bmatrix}
A^{[1]} \\ A^{[2]}
\end{bmatrix}.
\]
The $n\times n$ identity matrix is denoted by $I_{n}$.
By $0$ we denote the zero matrix, whose size should be clear in the context.
If $\mathcal{S},\mathcal{T}\subset \mathbb{C}^n$ are  subspaces with the same dimension, say $\ell$, we define the distance between $\mathcal{S}$ and $\mathcal{T}$ as
\begin{equation}\label{eq:def-distance}
\mathrm{dist}(\mathcal{S},\mathcal{T}):=\|P_\mathcal{S}-P_\mathcal{T}\|_2,
\end{equation}
where $P_\mathcal{S}$ and $P_\mathcal{T}$ are, respectively, orthogonal projectors onto $\mathcal{S}$ and $\mathcal{T}$.
It is well--known that the distance function \eqref{eq:def-distance} is a metric on the set of all $\ell$ dimensional subspaces of $\mathbb{C}^n$ \cite[Theorem 4.7]{Pete}. 
Given a matrix $Q\in\mathbb{C}^{n\times k}$, with $k\leq n$, we denote by $\mathrm{span}\{ Q \}$ the subspace spanned by the columns of $Q$.
We denote by $A^\dagger$ the Moore--Penrose pseudoinverse of a matrix $A$.
By $\epsilon$ we denote the unit roundoff, and we use the notation $O(\epsilon)$ for any quantity that is upper bounded by $\epsilon$ times a modest constant.

\section{The TOAR method and the stability analysis by Lu, Su and Bai}

We review in this section the main ideas underlying the TOAR method for computing an orthonormal basis of a second--order Krylov subspace, and the result of the stability analysis performed by Lu, Su and Bai \cite{TOAR}.
We essentially follow the presentation given in \cite{TOAR}.

We begin by recalling that the second--order Krylov subspace $\mathcal{G}_k(A,B;r_{-1},r_0)$ can be embedded in a Krylov subspace associated with the companion matrix
\begin{equation}\label{eq:companion}
C := \begin{bmatrix}
A & B \\ I_n & 0
\end{bmatrix}\in\mathbb{C}^{2n\times 2n}.
\end{equation}
Indeed, introducing the vector $\mathbf{v}_1:=\left[\begin{smallmatrix} r_0 \\ r_{-1}\end{smallmatrix}\right]$, the equality
\begin{equation}\label{eq:rel-SOKS-KS}
\mathcal{K}_k(C;\mathbf{v}_1):=\mathrm{span}\left\{ \mathbf{v}_1,C\mathbf{v}_1,\hdots, C^{k-1}\mathbf{v}_1 \right\}=\mathrm{span}\left\{\begin{bmatrix}r_0 \\ r_{-1} \end{bmatrix},\begin{bmatrix}r_1 \\ r_{0} \end{bmatrix}, \hdots, \begin{bmatrix}r_{k-1} \\ r_{k-2} \end{bmatrix} \right\}
\end{equation}
is immediately verified.
Hence, if $\mathbf{V}_k\in\mathbb{C}^{2n\times k}$ is a matrix whose columns form a  basis for $\mathcal{K}_k(C;\mathbf{v}_1)$ and writing
\[
\mathbf{V}_k=\begin{bmatrix}
V_k^{[1]} \\  V_k^{[2]}
\end{bmatrix} \quad \mbox{with} \quad  V_k^{[i]}\in\mathbb{C}^{n\times k}, \quad \mbox{for }i=1,2,
\] 
we readily obtain from \eqref{eq:rel-SOKS-KS} that
\begin{align}
&\mathrm{span}\left\{V_k^{[1]} \right\} = \mathrm{span}\left\{ r_0,r_1,\hdots,r_{k-1}\right\} \quad \mbox{and} \\
&\mathrm{span}\left\{V_k^{[2]} \right\} =\mathrm{span}\left\{ r_{-1},r_0,\hdots,r_{k-2}\right\},
\end{align}
and, therefore,
\[
\mathrm{span}\left\{ \begin{bmatrix} V_k^{[1]} & V_k^{[2]} \end{bmatrix} \right\}= \mathcal{G}_k(A,B;r_{-1},r_0).
\]
Thus, introducing $d_k:=\mathrm{dim}(\mathcal{G}_k(A,B;r_{-1},r_0))\leq k+1$ and denoting by $Q_k\in\mathbb{C}^{n\times d_k}$ a matrix whose columns form a basis for $\mathcal{G}_k(A,B;r_{-1},r_0)$, we can write
\begin{equation}\label{eq:two-levels}
\mathbf{V}_k=
\begin{bmatrix}
V_k^{[1]} \\  V_k^{[2]} 
\end{bmatrix}=
\begin{bmatrix}
Q_kU_k^{[1]} \\  Q_kU_k^{[2]} 
\end{bmatrix}=:
(I_2\otimes Q_k)\mathbf{U}_k,
\end{equation}
for some matrix $\mathbf{U}_k\in\mathbb{C}^{2d_k\times k}$. 
We will refer to \eqref{eq:two-levels} as a \emph{compact representation} of the matrix $\mathbf{V}_k$.
Furthermore, from \eqref{eq:two-levels}, we see that one (numerically expensive) possibility for computing a  basis for the subspace $\mathcal{G}_k(A,B;r_{-1},r_0)$ is by extracting it from a rank--revealing decomposition of the matrix 
\[
\begin{bmatrix} V_k^{[1]} & V_k^{[2]} \end{bmatrix}.
\]
The TOAR method provides a stable and computationally--efficient alternative for computing such a basis.

Before introducing the TOAR method, let us recall that a matrix $\mathbf{V}_k$ whose columns form an orthonormal basis for $\mathcal{K}_k(C;\mathbf{v}_1)$ can be computed in a numerically stable way by applying the Arnoldi algorithm to the companion matrix \eqref{eq:companion}.
Certainly, in exact arithmetic, running $k$ steps of the Arnoldi algorithm produces matrices satisfying
\begin{equation}\label{eq:Arnoldi-decomp}
\begin{bmatrix}
A & B \\ I_n & 0
\end{bmatrix}\mathbf{V}_k = \mathbf{V}_{k+1}\underline{H}_k,
\end{equation}
where $\underline{H}_k\in\mathbb{C}^{(k+1)\times k}$ is an upper--Hessenberg matrix  and the columns of $\mathbf{V}_{k}$ and $\mathbf{V}_{k+1}=\begin{bmatrix} \mathbf{V}_k & \mathbf{v}_{k+1}\end{bmatrix}$ form orthonormal bases for $\mathcal{K}_k(C;\mathbf{v}_1)$ and $\mathcal{K}_{k+1}(C;\mathbf{v}_1)$, respectively.
We will refer to \eqref{eq:Arnoldi-decomp} as an \emph{Arnoldi decomposition}.

By combining the compact representation \eqref{eq:two-levels} with the Arnoldi decomposition \eqref{eq:Arnoldi-decomp}, we get the decomposition
\begin{equation}\label{eq:TOAR-decomp}
\begin{bmatrix}
A & B \\ I_n & 0
\end{bmatrix}(I_2\otimes Q_k)\mathbf{U}_k = (I_2\otimes Q_{k+1})\mathbf{U}_{k+1}\underline{H}_k,
\end{equation}
which will be referred to as a \emph{TOAR decomposition}.
The TOAR method is a memory--efficient variant of the  Arnoldi method \cite{Kressner-Roman,Q-Arnoldi,CORK} applied to the companion matrix \eqref{eq:companion} for computing \eqref{eq:TOAR-decomp}.
By exploiting the compact representation of $\mathbf{V}_k$ in \eqref{eq:two-levels}--\eqref{eq:TOAR-decomp}, it computes matrices $Q_k$ and $\mathbf{U}_k$ with orthonormal columns and the Hessenberg matrix $\underline{H}_k$, without forming explicitly the matrix $\mathbf{V}_k$.
Notice in passing that the orthonormality of the columns of $\mathbf{U}_k$ and $Q_k$ implies the orthonormality of the columns of $\mathbf{V}_k$.
We refer the reader to \cite{Roman2016,TOAR} for implementation details, and to \cite{Kressner-Roman,MP,CORK} for extensions of the TOAR algorithm to other companion matrices.

In the presence of finite precision arithmetic, the TOAR method is numerically stable \cite{TOAR,MP}, provided that the orthogonalization steps  have been properly carried out \cite{ortho}.
More precisely, the computed matrices have, up to working precision, orthonormal columns and, together with the computed matrix $\underline{H}_k$, satisfy 
\begin{equation}\label{eq:residual}
R = 
\begin{bmatrix}
A & B\\
I_n & 0
\end{bmatrix}
(I_2\otimes Q_k)\mathbf{U}_k -
(I_2\otimes Q_{k+1})\mathbf{U}_{k+1}\underline{H}_k,
\end{equation}
for some matrix $R\in\mathbb{C}^{2n\times k}$ with $\|R\|_2=O(\epsilon)\|C\|_2$.
Then, it is standard to show from \eqref{eq:residual} that the columns of the computed matrix $\mathbf{V}_k=(I_2\otimes Q_k)\mathbf{U}_k$ form a basis for a Krylov subspace $\mathcal{K}_k(C+E;\mathbf{v}_1)$, for some matrix $E$ with $\|E\|_2=O(\epsilon)\|C\|_2$ \cite{TOAR,Pete2}.
 However, the perturbation $E$ destroys the companion matrix structure, i.e., the zero and identity blocks of $C$ are not present  in $C+E$.
 Therefore, it is not clear whether or not the columns of the computed matrix $Q_k$ span a basis for some second--order Krylov subspace associated with  perturbed matrices $A+\Delta A$ and $B+\Delta B$.
 This problem, left open in \cite{TOAR}, is solved in the following section.

\section{Mixed forward--backward stability of the TOAR method in computing an orthonormal basis of $\mathcal{G}_k(A,B;r_{-1},r_0)$}\label{sec:analysis}

The starting point is the residual \eqref{eq:residual}, and our goal is to throw it back onto the matrices $A$ and $B$.
This is done in Theorem \ref{thm:main}, which is one of our main results.
As a corollary, we will obtain a mixed forward--backward stability result for the TOAR method in Corollary \ref{cor:stability}.
The proof of Theorem \ref{thm:main} is postponed to the end of the section.

\begin{theorem}\label{thm:main}
Let $A,B\in\mathbb{C}^{n\times n}$ and let $C$ be the companion matrix \eqref{eq:companion}.
Let $\underline{H}_k\in\mathbb{C}^{(k+1)\times k}$, and let $Q_k\in\mathbb{C}^{n\times d_k}$, $Q_{k+1}=\begin{bmatrix} Q_k & q_{k+1}\end{bmatrix}\in\mathbb{C}^{n\times (d_k+1)}$, with $d_k\leq k+1$, be full--column--rank matrices.
Let
 \[
\mathbf{U}_k=\begin{bmatrix} U_k^{[1]} \\  U_k^{[2]} \end{bmatrix}\in\mathbb{C}^{2d_k\times k} \quad \mbox{and} \quad \mathbf{U}_{k+1}=\begin{bmatrix} U_{k+1}^{[1]} \\  U_{k+1}^{[2]} \end{bmatrix}=
\left[\begin{array}{cc} \begin{matrix} U_k^{[1]} & x_k \\ 0 & \beta_k \end{matrix} \\ \begin{matrix} U_k^{[2]} & y_k \\ 0 & 0 \end{matrix} \end{array}\right] \in\mathbb{C}^{2(d_k+1)\times (k+1)}
\]
 be also full--column--rank matrices.
 Let $R$  be the  residual \eqref{eq:residual}, let $E=-R\mathbf{U}_k^\dagger(I_2\otimes Q_k^\dagger)$, and let $\mathcal{S}_k\subseteq\mathbb{C}^n$ be the subspace spanned by the columns of $Q_{k}$.
If $\|E\|_2 < 1$, then there exists a $d_k$--dimensional second--order Krylov subspace $\mathcal{G}_k(A+\Delta A,B+\Delta B;\widetilde{r}_{-1},\widetilde{r}_0)$ such that 
\begin{equation}\label{eq:distance}
\mathrm{dist}\left(\mathcal{S}_k,\mathcal{G}_k(A+\Delta A,B+\Delta B;\widetilde{r}_{-1},\widetilde{r}_0)\right)\leq \frac{\|E\|_2}{1-\|E\|_2},
\end{equation}
for some vectors $\widetilde{r}_{-1}$ and $\widetilde{r}_0$, and some matrices $\Delta A$ and $\Delta B$ with
\begin{equation}\label{eq:bound-A}
\|\Delta A\|_2 \leq \|E\|_2+\frac{\|E\|_2(1+\|E\|_2)}{1-\|E\|_2},
\end{equation}
and 
\begin{equation}\label{eq:bound-B}
\|\Delta B\|_2 \leq \max\{1,\|A\|_2,\|B\|_2\}\left(
\|E\|_2(2+\|E\|_2)+\frac{\|E\|_2(1+\|E\|_2)^2}{1-\|E\|_2}
\right).
\end{equation}
\end{theorem}
\begin{remark}
The structure of the matrix $\mathbf{U}_{k+1}$ in Theorem \ref{thm:main}  is imposed to make the matrix compatible with a TOAR decomposition \cite[Lemma 31]{TOAR}.
In particular, this is the structure of the computed matrix $\mathbf{U}_{k+1}$ by the TOAR method in floating point arithmetic \cite{TOAR}.
This structure for $\mathbf{U}_{k+1}$ will be assumed throughout the rest of the section.
\end{remark}

As an immediate corollary of Theorem \ref{thm:main}, we obtain the following mixed forward--backward stability result for the TOAR method. 
\begin{corollary}\label{cor:stability}
Let $Q_{k}\in\mathbb{C}^{n\times d_k}$ be the matrix obtained by the TOAR method run in a computer with unit roundoff equal to $\epsilon$, i.e., the computed quantities satisfy \eqref{eq:residual} with $\|R\|_2=O(\epsilon)\|C\|_2$.
Assume that $Q_k$ has full column rank, and let $\mathcal{S}_k$ be the subspace spanned by the columns of $Q_{k}$.
Then, to first order in $\epsilon$,  there exists a $d_k$--dimensional second--order Krylov subspace $\mathcal{G}_k(A+\Delta A,B+\Delta B;\widetilde{r}_{-1},\widetilde{r}_0)$ such that
\begin{equation}\label{eq:distance-TOAR}
\mathrm{dist}\left(\mathcal{S}_k,\mathcal{G}_k(A+\Delta A,B+\Delta B;\widetilde{r}_{-1},\widetilde{r}_0)\right)=O(\epsilon)\max\{1,\|A\|_2,\|B\|_2 \},
\end{equation}
 for some vectors $\widetilde{r}_{-1}$ and $\widetilde{r}_0$, and some matrices $\Delta A$ and $\Delta B$ with
\begin{align}\label{eq:DeltaA-DeltaB}
\begin{split}
&\|\Delta A\|_2=O(\epsilon)\max\{1,\|A\|_2,\|B\|_2 \}, \quad \mbox{and}\\
&\|\Delta B\|_2=O(\epsilon) \left(\max\{1,\|A\|_2,\|B\|_2 \}\right)^2.
\end{split}
\end{align}
In words, the columns of the computed matrix $Q_{k}$ span a subspace  close to a  second--order Krylov subspace associated with matrices $A+\Delta A$ and $B+\Delta B$.
\end{corollary}
\begin{proof}
The error analysis by Lu, Su and Bai \cite{TOAR} shows that the columns of the matrices $Q_k$ and $\mathbf{U}_k$ computed by the TOAR method are both well--conditioned bases of the subspaces they span, that is,
\[
\|Q_k\|_2\|Q_k^\dagger\|_2 = 1 + O(\epsilon) \quad \mbox{and} \quad \|\mathbf{U}_k\|_2\|\mathbf{U}_k^\dagger\|_2 = 1 + O(\epsilon),
\]
and that the norm of the residual \eqref{eq:residual} for the computed matrices is a modest multiple of the unit roundoff times the norm of the companion matrix. 
Setting $E=-R\mathbf{U}_k^\dagger(I_2\otimes Q_k^\dagger)$,  we obtain that Theorem \ref{thm:main} holds with
\begin{equation}\label{eq:norm_E}
\|E\|_2 =O(\epsilon)\|C\|_2 =O(\epsilon)\max\{1,\|A\|_2,\|B\|_2 \}.
\end{equation}
To finish the proof, it suffices to notice that \eqref{eq:distance}, \eqref{eq:bound-A} and \eqref{eq:bound-B}, together with \eqref{eq:norm_E}, imply \eqref{eq:distance-TOAR} and \eqref{eq:DeltaA-DeltaB}  to first order in $\epsilon$.
\end{proof}

Paraphrasing Higham \cite{Higham}, Theorem \ref{thm:main} tells us that, as long as the norms of the matrices $A$ and $B$ are not too large or too small, the basis computed by the TOAR method is ``almost the right answer for almost the right data''.
Hence, in this situation, TOAR is  numerically stable in computing orthonormal bases of second--order Krylov subspaces.
When the norms of $A$ and $B$ are very large or very small, one could consider scaling the problem for improving the stability properties of TOAR.
This is considered in Section \ref{sec:scaling}.

 The proof of Theorem \ref{thm:main} requires several technical results that we state in the following lemmas.
Lemma \ref{lemma:back-to-C} projects the  residual \eqref{eq:residual} back on the companion matrix \eqref{eq:companion}.
\begin{lemma}\label{lemma:back-to-C}
If $R$ denotes the  residual \eqref{eq:residual}, then the matrix $E=-R\mathbf{U}_k^\dagger(I_2\otimes Q_k^\dagger)$ satisfies the decomposition
\begin{equation}\label{eq:TOAR-pert}
\Bigg{(} \begin{bmatrix}
A & B\\
I_n & 0
\end{bmatrix}
+\underbrace{\begin{bmatrix} E_{11} & E_{12} \\ E_{21} & E_{22} \end{bmatrix}}_{=E}
\Bigg{)}
(I_2\otimes Q_k)\mathbf{U}_k =
(I_2\otimes Q_{k+1})\mathbf{U}_{k+1}\underline{H}_k,
\end{equation}
where $E$ in \eqref{eq:TOAR-pert} has been partitioned conformably to the partition of the companion matrix \eqref{eq:companion}.
In other words, the matrices satisfy an exact TOAR decomposition for a perturbed  matrix $C+E$.
\end{lemma}
\begin{proof}
It is immediately verified that $E(I_2\otimes Q_k)\mathbf{U}_k=-R$. 
\end{proof}
 
 The decomposition \eqref{eq:TOAR-pert} is an exact TOAR decomposition, but the matrix $C+E$ is not a companion matrix.
 Thus, we cannot yet associated $Q_k$ with a second--order Krylov subspace.
 Nevertheless, Lemma \ref{lemma:recover-comp} shows that, as long as the norm of the perturbation $E$ in Lemma \ref{lemma:back-to-C} is small enough, the companion structure of the perturbed companion matrix in \eqref{eq:TOAR-pert} can be recovered via a similarity transformation.
\begin{lemma}\label{lemma:recover-comp}
Let $C$ be the companion matrix in \eqref{eq:companion}.
If $E=\left[ \begin{smallmatrix} E_{11} & E_{12} \\ E_{21} & E_{22}\end{smallmatrix}\right]$, where $E_{ij}\in\mathbb{C}^{n\times n}$, is a matrix such that $\|E_{21}\|_2 < 1$, then
\[
\begin{bmatrix}
I_n & (I_n+E_{21})^{-1}E_{22}\\
0 & (I_n+E_{21})^{-1}
\end{bmatrix}
(C+E)=
\begin{bmatrix}
A+\Delta A & B+\Delta B \\ I_n & 0 
\end{bmatrix}
\begin{bmatrix}
I_n & (I_n+E_{21})^{-1}E_{22}\\
0 & (I_n+E_{21})^{-1}
\end{bmatrix},
\]
where the matrices $\Delta A$ and $\Delta B$ are equal to
\begin{align}
\label{eq:DeltaA}&\Delta A = E_{11}+(I_n+E_{21})^{-1}E_{22}(I_n+E_{21}), \quad \mbox{and} \\
\label{eq:DeltaB}&\Delta B = BE_{21}+E_{12}(I_n+E_{21})-(A+E_{11})(I_n+E_{21})^{-1}E_{22}(I_n+E_{21}).
\end{align}
In words, the companion structure can be recovered via a similarity transformation close to the identity.
\end{lemma}
\begin{proof}
The condition $\|E_{21}\|_2<1$ guarantees the nonsingularity of the matrix $I_n+E_{21}$. 
Then, the result can be easily checked by performing directly the matrix multiplications.
\end{proof}

\begin{remark}\label{remark}
The idea of recovering the ``companion structure'' of a perturbed companion matrix by using transformations close to the identity as in the proof of Lemma \ref{lemma:recover-comp} has appeared several times in the context of studying the numerical stability of solving polynomial eigenvalue problems by linearization \cite{NN,NP,VanDooren}.
\end{remark}

Applying Lemma \ref{lemma:recover-comp} to the perturbed companion matrix in the TOAR decomposition \eqref{eq:TOAR-pert}, we obtain
\begin{align}\label{eq:Arnoldi-perturb}
\begin{split}
\begin{bmatrix}
A+\Delta A & B+\Delta B \\ I_n & 0 
\end{bmatrix}
&\begin{bmatrix}
I_n & (I_n+E_{21})^{-1}E_{22}\\
0 & (I_n+E_{21})^{-1}
\end{bmatrix}
(I_2\otimes Q_k)\mathbf{U}_k\\ &\hspace{1cm} = 
\begin{bmatrix}
I_n & (I_n+E_{21})^{-1}E_{22}\\
0 & (I_n+E_{21})^{-1}
\end{bmatrix}
(I_2\otimes Q_{k+1})\mathbf{U}_{k+1}
\underline{H}_k,
\end{split}
\end{align}
where the matrices $\Delta A$ and $\Delta B$ are defined in \eqref{eq:DeltaA}--\eqref{eq:DeltaB}.
Then, introducing the new matrices
\begin{equation}\label{eq:W}
\mathbf{W}_i =
\begin{bmatrix}
W_{i}^{[1]}\\ W_{i}^{[2]}
\end{bmatrix} :=
\begin{bmatrix}
Q_iU_{i}^{[1]}+(I_n+E_{21})^{-1}E_{22}Q_i U_{i}^{[2]}\\
(I_n+E_{21})^{-1}Q_i U_{i}^{[2]}
\end{bmatrix}, \quad \mbox{for }i=k,k+1,
\end{equation}
the decomposition \eqref{eq:Arnoldi-perturb} becomes
\begin{equation}\label{eq:new-Arnoldi-decomp}
\begin{bmatrix}
A+\Delta A & B+\Delta B \\ I_n & 0 
\end{bmatrix}\begin{bmatrix}
W_{k}^{[1]}\\ W_{k}^{[2]}
\end{bmatrix}
=\begin{bmatrix}
W_{k+1}^{[1]}\\ W_{k+1}^{[2]}
\end{bmatrix}\underline{H}_k.
\end{equation}

The compact representation of the matrix $\mathbf{V}_k=(I_2\otimes Q_k)\mathbf{U}_k$ is destroyed after  premultiplying $\mathbf{V}_k$ by the matrix $\left[\begin{smallmatrix}
I_n & (I_n+E_{21})^{-1}E_{22}\\
0 & (I_n+E_{21})^{-1}
\end{smallmatrix}\right]$.
Nevertheless, the obtained decomposition \eqref{eq:new-Arnoldi-decomp} is an exact Arnoldi decomposition for a companion matrix.
Hence, we obtain
\begin{equation}\label{eq:SOKS-W}
\mathrm{span}\left\{\begin{bmatrix} W_{k}^{[1]} & W_{k}^{[2]} \end{bmatrix} \right\}=
 \mathcal{G}_k(A+\Delta A,B+\Delta B;\widetilde{r}_{-1},\widetilde{r}_0),
\end{equation}
for some vectors $\widetilde{r}_{-1},\widetilde{r}_0$ and some  matrices $A+\Delta A$ and $B+\Delta B$.
In Lemma \ref{lemma:basis}, we obtain a  basis for the subspace  \eqref{eq:SOKS-W}.

\begin{lemma}\label{lemma:basis}
Let \eqref{eq:TOAR-pert} be the TOAR decomposition for a slightly perturbation of a companion matrix $C$ as in \eqref{eq:companion}, and let
 $\mathbf{W}_i$, with $i=k,k+1$, be the matrices defined in \eqref{eq:W}.
 If $\|E_{21}\|_2<1$, then
\[
\mathrm{span}\left\{\begin{bmatrix} W_{k}^{[1]} & W_{k}^{[2]} \end{bmatrix} \right\}= \mathrm{span}\left\{ (I_n+E_{21})^{-1}Q_{k} \right\},
\]
and $\mathrm{dim}\left(\mathrm{span}\left\{ (I_n+E_{21})^{-1}Q_{k} \right\}\right)=d_k$.
\end{lemma}
\begin{proof}
From Lemma \ref{lemma:recover-comp} and the hypothesis $\|E_{21}\|_2<1$, we obtain that the matrices $\mathbf{W}_i$, with $i=k,k+1$, satisfy the Arnoldi decomposition \eqref{eq:new-Arnoldi-decomp}. 
Examining the bottom block of $\mathbf{W}_k$ in \eqref{eq:W}, we readily obtain that $\mathrm{span}\{ W_{k}^{[2]}\}\subseteq\mathrm{span}\{ (I_n+E_{21})^{-1}Q_{k}\}$.
Further, from the bottom block of \eqref{eq:new-Arnoldi-decomp}, together with \eqref{eq:W}, we obtain
\[
W_k^{[1]}=(I_n+E_{21})^{-1}Q_{k+1}U_{k+1}^{[2]}\underline{H}_k.
\]
Then, from the fact that the matrix $U_{k+1}^{[2]}$ is of the form (recall Remark \ref{remark})
\[
U_{k+1}^{[2]} = 
\begin{bmatrix}
U_k^{[2]} & y_k \\
0 & 0
\end{bmatrix},
\]
for some vector $y_k$, we  obtain that the columns of $W_k^{[1]}$ are linear combinations of only the first $d_k$ columns of $(I_n+E_{21})^{-1}Q_{k+1}$, i.e., the columns of $(I_n+E_{21})^{-1}Q_k$.
Therefore, $\mathrm{span}\{ W_{k}^{[1]}\}\subseteq\mathrm{span}\{ (I_n+E_{21})^{-1}Q_{k}\}$.
Finally, it is clear that $d_k=\mathrm{rank}(Q_k)=\mathrm{rank}((I_n+E_{21})^{-1}Q_k)$.
\end{proof}

The last auxiliary result for the proof of Theorem \ref{thm:main} is Lemma \ref{lemma:distance}, which  shows how a subspace spanned by the columns of a matrix $A$ behaves under multiplicative perturbations of the matrix $A$.
\begin{lemma}{\rm \cite[Theorem 3.3]{Pseudo}}\label{lemma:distance}
Let $A\in\mathbb{C}^{m\times n}$ and $\widetilde{A}=(I_m+E)A\in\mathbb{C}^{m\times n}$, where $(I_m+E)\in\mathbb{C}^{m\times m}$ is nonsingular.
Then,
\[
\mathrm{dist}(\mathcal{S},\widetilde{\mathcal{S}})\leq
\min\{\|E\|_2,\|(I_m+E)^{-1}E\|_2 \},
\]
where $\mathcal{S}$ and $\widetilde{\mathcal{S}}$ are the subspaces spanned, respectively, by the columns of the matrices $A$ and $\widetilde{A}$.
\end{lemma}

We are finally in a position to prove Theorem \ref{thm:main}.

\begin{proof} {\bf  (of Theorem \ref{thm:main})}. 
Recall that $E=-R\mathbf{U}_k^\dagger(I_2\otimes Q_k^\dagger)$.
Since $\|E\|_2<1$ and, thus, $\|E_{21}\|_2<1$, we obtain from Lemmas \ref{lemma:back-to-C} and \ref{lemma:recover-comp} that the matrices $\mathbf{W}_i$, with $i=k,k+1$,  defined in \eqref{eq:W} satisfy the Arnoldi decomposition \eqref{eq:new-Arnoldi-decomp}.
Therefore, \eqref{eq:SOKS-W} is  a $d_k$--dimensional second--order Krylov subspace associated with matrices $A+\Delta A$ and $B+\Delta B$, with $\Delta A$ and $\Delta B$ as in \eqref{eq:DeltaA} and \eqref{eq:DeltaB}, respectively.
Furthermore, from \eqref{eq:DeltaA}, we obtain
\[
\|\Delta A\|_2 \leq \|E\|_2+\frac{\|E\|_2(1+\|E\|_2)}{1-\|E\|_2},
\] 
and from \eqref{eq:DeltaB}, we obtain
\[
\|\Delta B\|_2 \leq \max\{1,\|A\|_2,\|B\|_2\}\left(
\|E\|_2(2+\|E\|_2)+\frac{\|E\|_2(1+\|E\|_2)^2}{1-\|E\|_2}
\right),
\]
where we have used $\|(I_n+E_{21})^{-1}\|_2\leq (1-\|E_{21}\|_2)^{-1}$ and $\|E_{ij}\|_2\leq \|E\|_2$, for $i,j=1,2$, for obtaining both upper bounds.

From Lemma \ref{lemma:basis}, we obtain that the columns of $\widetilde{Q}_k:=(I_n+E_{21})^{-1}Q_{k}$ form a basis for  the second--order Krylov subspace \eqref{eq:SOKS-W}.
Let $\widetilde{\mathcal{S}}_k$ be the subspace spanned by the columns of $\widetilde{Q}_k$.
To finish the proof of Theorem \ref{thm:main}, it suffices to  bound the distance between the subspaces $\mathcal{S}_k$ and $\widetilde{\mathcal{S}}_k$ from above.
Writing $Q_k=(I_n+E_{21})\widetilde{Q}_k$, we immediately obtain from Lemma \ref{lemma:distance} 
\[
\mathrm{dist}(\mathcal{S}_k,\widetilde{\mathcal{S}}_k)\leq
\min\{\|E_{21}\|_{2},\|(I_n+E_{21})^{-1}E_{21}\|_2 \}\leq 
\frac{\|E\|_2}{1-\|E\|_2},
\]
and the proof is completed.
\end{proof}

\section{Scaling the quadratic problem for improving the numerical stability of the TOAR procedure}\label{sec:scaling}
In this section, we study the effect of scaling the original quadratic problem on the stability of the TOAR procedure.

A reasonable definition of a stable algorithm for computing an orthonormal basis of a second--order Krylov subspace $\mathcal{G}_k(A,B;r_{-1},r_0)$ is to require that the algorithm  computes an  exact basis of (or, up to machine precision, a basis close to)  a second--order Krylov subspace associated with matrices $A+\Delta A$ and $B+\Delta B$, with $\Delta A$ and $\Delta B$ satisfying
\begin{equation*}\label{eq:ideal-normwise}
\max\{\|\Delta A\|_2,\|\Delta B\|_2 \}= O(\epsilon) \max\{ \|A\|_2,\|B\|_2\},
\end{equation*}
or the more stringent condition
\begin{equation*}\label{eq:ideal-coeffwise}
\max\left\{ \frac{\|\Delta A\|_2}{\|A\|_2}, \frac{\|\Delta B\|_2}{\|B\|_2}\right\}= O(\epsilon).
\end{equation*}
In the former case, we would say that the algorithm is normwise stable, and in the latter,  coefficientwise stable.

From Corollary \ref{cor:stability}, we see that the TOAR method fails to be stable in two situations, namely, when the norms of $A$ and $B$ are much smaller or much bigger than 1.
We show in this section that in any of these situations the computed subspace could gain in accuracy when using the TOAR method on an appropriate scaling of the quadratic problem.

Scaling the quadratic problem consists in replacing the matrices $A$ and $B$ by the matrices $A_\alpha:=\alpha A$ and $B_\alpha:=\alpha^2 B$, where $\alpha$ is a nonzero positive real number.
This operation is reflected  as a scaling of the eigenvalues of the quadratic eigenvalue problem or as a scaling of the frequencies of the transfer function of the second--order dynamical system.
The companion matrix associated with $A_\alpha$ and $B_\alpha$ is 
\begin{equation}\label{eq:scaling-C}
C_\alpha := \begin{bmatrix}
A_\alpha & B_\alpha \\ I_n & 0
\end{bmatrix}=\alpha
\begin{bmatrix}
I_n & 0 \\ 0 & \alpha^{-1}I_n
\end{bmatrix}
\begin{bmatrix}
A & B \\ I_n & 0
\end{bmatrix}
\begin{bmatrix}
I_n & 0 \\ 0 & \alpha I_n
\end{bmatrix}.
\end{equation}

According to the analysis in Section \ref{sec:analysis}, applying the TOAR procedure to the companion matrix $C_\alpha$ produces a computed matrix $Q_k$ satisfying a TOAR decomposition (recall Corollary \ref{cor:stability}) of the form
\begin{equation}\label{eq:TOAR-scaling}
\begin{bmatrix}
A_\alpha + \Delta A_\alpha & B_\alpha + \Delta B_\alpha \\ I_n & 0 \end{bmatrix}(I_2\otimes \widetilde{Q}_k)\widetilde{\mathbf{U}}_k = 
(I_2\otimes \widetilde{Q}_{k+1})\widetilde{\mathbf{U}}_{k+1}\underline{H}_k,
\end{equation}
for some matrices $\widetilde{\mathbf{U}}_{k}$ and $\widetilde{\mathbf{U}}_{k+1}$, and some matrices $\Delta A_\alpha$ and $\Delta B_\alpha$ with
\begin{align*}\label{eq:DeltaA-DeltaB-alpha}
\begin{split}
&\|\Delta A_\alpha\|_2= O(\epsilon)\max\{1,\|A_\alpha \|_2,\|B_\alpha \|_2 \}, \quad \mbox{and} \\
&\|\Delta B_\alpha \|_2=O(\epsilon) \left(\max\{1,\|A_\alpha \|_2,\|B_\alpha\|_2 \}\right)^2,
\end{split}
\end{align*}
and where $\widetilde{Q}_k$ is a matrix such that
\[
\mathrm{dist}(\mathrm{span}(\widetilde{Q}_k),\mathrm{span}(Q_k))=O(\epsilon) \max\{1,\|A_\alpha \|_2,\|B_\alpha \|_2 \}.
\]

Undoing the scaling by using \eqref{eq:scaling-C}, we obtain from \eqref{eq:TOAR-scaling} the perturbed TOAR decomposition
\[
\begin{bmatrix}
A + \alpha^{-1} \Delta A_\alpha & B + \alpha^{-2}\Delta B_\alpha \\ I_n & 0 \end{bmatrix} (I_2\otimes \widetilde{Q}_k)\widehat{\mathbf{U}}_k = 
(I_2\otimes \widetilde{Q}_{k+1})\widehat{\mathbf{U}}_{k+1}\underline{\widehat{H}}_k,
\]
where one $\alpha$ has been absorbed by $\underline{H}_k$ and the other two $\alpha$'s have been absorbed by the bottom blocks of $\widetilde{\mathbf{U}}_{k}$ and $\widetilde{\mathbf{U}}_{k+1}$.
We conclude that the computed matrix $Q_k$ is such that the subspace spanned by its columns is within a distance 
\begin{equation}\label{eq:dist-alpha}
O(\epsilon) \max\{1,\alpha \|A\|_2,\alpha^2 \|B\|_2 \}
\end{equation}
 of a second--order Krylov subspace associated with matrices $A+\alpha^{-1}\Delta A_\alpha=:A+\Delta A$ and $B+\alpha^{-2}\Delta B_\alpha =: B+\Delta B$ with
\begin{align}\label{eq:DeltaA-DeltaB-alpha2}
\begin{split}
&\|\Delta A\|_2= O(\epsilon)\alpha^{-1}\max\{1,\alpha\|A\|_2,\alpha^2\|B\|_2 \}, \quad \mbox{and}\\
&\|\Delta B \|_2=O(\epsilon) \alpha^{-2}\left(\max\{1,\alpha\|A \|_2,\alpha^2\|B\|_2 \}\right)^2.
\end{split}
\end{align}
Hence, the problem of choosing an optimal scaling parameter $\alpha$ for improving the stability of the TOAR procedure is reduced to the problem of minimizing \eqref{eq:dist-alpha} and \eqref{eq:DeltaA-DeltaB-alpha2} over $\alpha\in\mathbb{R}^+$.

We attempt to find a good choice of the scaling parameter $\alpha$  by minimizing the function
\[
f(\alpha):=\frac{1+\|A\|_2\alpha+\|B\|_2\alpha^2}{\alpha},
\]
which is essentially equivalent to minimizing \eqref{eq:DeltaA-DeltaB-alpha2}.
We find that $\alpha_{\rm opt}:=\|B\|_2^{-1/2}$ is a local minimum of $f(\alpha)$. 
This scaling parameter corresponds to the scaling  introduced by Fan, Lin and Van Dooren \cite{FLD} for improving the backward stability of solving quadratic matrix polynomials by linearization.

We summarize in Theorem \ref{thm:scaling} the effect of scaling the quadratic problem with $\alpha=\alpha_{\rm opt}$  on the stability of computing an orthonormal basis of the second--order Krylov subspace $\mathcal{G}_k(A,B;r_{-1},r_0)$ by applying the TOAR method to the scaled companion matrix $C_{\alpha_{\rm opt}}$.
\begin{theorem}\label{thm:scaling}
Let $A,B\in\mathbb{C}^{n\times n}$, let $\alpha_{\rm opt}=\|B\|_2^{-1/2}$, and let $Q_k\in\mathbb{C}^{n\times d_k}$ be the computed matrix by the TOAR method applied to the scaled companion matrix $C_{\alpha_{\rm opt}}$ in a computer with unit roundoff equal to $\epsilon$.
Then, the following statements hold.
\begin{itemize}
\item[\rm (i)] If $\|A\|_2 \leq \|B\|_2^{1/2}$, then the subspace spanned by the columns of $Q_k$ is within a distance $O(\epsilon)$ of a second--order Krylov subspace associated with matrices $A+\Delta A$ and $B+\Delta B$ such that
\[
\|\Delta A\|_2=O(\epsilon)\|B\|_2^{1/2}\quad \mbox{and} \quad \frac{\|\Delta B\|_2}{\|B\|_2} = O(\epsilon) .
\]
\item[\rm (ii)] If $\|A\|_2 \approx \|B\|_2^{1/2}$, then the subspace spanned by the columns of $Q_k$ is within a distance $O(\epsilon)$ of a second--order Krylov subspace associated with matrices $A+\Delta A$ and $B+\Delta B$ such that
\[
\max\left\{ \frac{\|\Delta A\|_2}{\|A\|_2}, \frac{\|\Delta B\|_2}{\|B\|_2}\right\}= O(\epsilon).
\]
\item[\rm (iii)] If $\|A\|_2>\|B\|_2^{1/2}$, then the subspace spanned by the columns of $Q_k$ is within a distance $O(\epsilon)\|B\|_2^{-1/2}\|A\|_2$ of a second--order Krylov subspace associated with matrices $A+\Delta A$ and $B+\Delta B$ such that
\[
\frac{\|\Delta A\|_2}{\|A\|_2}=O(\epsilon) \quad \mbox{and}\quad \|\Delta B\|_2 = O(\epsilon) \|A\|_2^2.
\]
\end{itemize}
\end{theorem}
\begin{proof}
The results readily follow from \eqref{eq:dist-alpha} and \eqref{eq:DeltaA-DeltaB-alpha2} with $\alpha=\alpha_{\rm opt}$.
\end{proof}

From Theorem \ref{thm:scaling}, we obtain  that the TOAR procedure applied to the scaled companion matrix $C_{\alpha_{\rm opt}}$ is normwise stable in computing an orthonormal basis of $\mathcal{G}_k(A,B;r_{-1},r_0)$ in the case where $\|A\|_2 \leq \|B\|_2^{1/2}$.
When $\|A\|_2 \approx \|B\|_2^{1/2}$, the method is actually coefficientwise stable.
However, when $\|A\|_2\gg\|B\|_2^{1/2}$, the scaling with $\alpha=\alpha_{\rm opt}$ does not resolve the stability issues of the TOAR method.
In the language of quadratic matrix polynomials, this situation corresponds to the so called heavily damped quadratic matrix polynomials \cite{quadeig}, and it is still an open problem to devise  simple scaling strategies for those.
\begin{remark}
When $\|A\|_2\gg\|B\|_2^{1/2}$, we could also consider the scaling with parameter $\alpha=\|A\|_2^{-1}$.
In this case, we would obtain that the subspace spanned by the columns of the computed matrix $Q_k$ by the TOAR method applied to $C_\alpha$ is withing a distance $O(\epsilon)$ of a second--order Krylov subspace associated with matrices $A+\Delta A$ and $B+\Delta B$ such that 
\[
\frac{\|\Delta A\|_2}{\|A\|_2}=O(\epsilon) \quad \mbox{and}\quad \|\Delta B\|_2 = O(\epsilon) \|A\|_2^2,
\]
which is an improvement over part--{\rm (iii)} in Theorem \ref{thm:scaling} and Corollary \ref{cor:stability}.
\end{remark}

\section{Conclusions}

Second--order Krylov subspace projection methods combined with the TOAR procedure have demonstrated superior numerical results over the standard approaches based on linearization for the solution of quadratic eigenvalue problems and for model order reduction of second--order dynamical systems.
In this work, we have shown that the computed basis by the TOAR method for the subspace $\mathcal{G}_k(A,B;r_{-1},r_0)$ is, up to machine precision, a second--order Krylov subspace associated with nearby matrices $A+\Delta A$ and $B+\Delta B$, providing to the observed numerical superiority a solid theoretical foundation.
We have also considered the effect of scaling the original quadratic problem on the numerical stability of the TOAR method in computing an orthonormal basis of $\mathcal{G}_k(A,B;r_{-1},r_0)$, and showed that in many situations the TOAR procedure applied to a scaled companion matrix is normwise, or even coefficientwise stable, in computing such a basis. 

%


\begin{thebibliography}{1}


\bibitem{Bai2005}
Z.\ Bai, and Y.\ Su.
\newblock Dimension reduction of large--scale second--order dynamical systems via second--order Arnoldi method.
\newblock {\em SIAM J. Sci. Comput.}, 26, pp. 1692--1709 (2005).

\bibitem{SOAR}
Z.\ Bai, and Y.\ Su.
\newblock SOAR: A second--order Arnoldi method for the solution of the quadratic eigenvalue problem.
\newblock {\em SIAM J. Matrix Anal. Appl.}, 26, pp. 640--659 (2005).

\bibitem{Roman2016}
C.\ Campos, and J.\ E.\ Roman.
\newblock Parallel Krylov solvers for the polynomial eigenvalue problem in SLEPc.
\newblock {\em SIAM J. Sci. Comput.}, 38(5), pp. S385--S411 (2016).

\bibitem{Pseudo}
N.\ Castro--Gonzalez, F.\ M.\ Dopico, and J.\ M.\ Molera.
\newblock Multiplicative perturbation theory of the Moore--Penrose inverse and the least squares problem.
\newblock {Linear Algebra Appl.}, 503, pp. 1--25 (2016).

\bibitem{ortho}
J.\ Daniel, W.\ B.\ Gragg, L.\ Kaufman, and G.\ W.\ Stewart.
\newblock Reorthogonalization and stable algorithms for updating the Gram--Schmidt QR factorization.
\newblock {\em Math. Comp.}, 30, pp. 772--795 (1976).

\bibitem{FLD}
H.--Y.\ Fan, W.--W.\ Lin, and P. Van Dooren.
\newblock Normwise scaling of second order polynomial matrices.
\newblock {\em SIAM J. Matrix Anal. Appl.}, 26(1), pp. 252--256 (2005).

\bibitem{quadeig}
S.\ Hammarling, C.\ Munro, and F.\ Tisseur.
\newblock An Algorithm for the complete solution of quadratic eigenvalue problems.
\newblock {\em ACM Transactions on Mathematical Software}, 39(3), pp. 18:1--18:19 (2013).
 
\bibitem{Higham}
N.\ J.\ Higham.
\newblock {\em Accuracy and Stability of Numerical Algorithms}.
\newblock 1st edition, SIAM, Philadelphia, 1996.

\bibitem{q-Ritz}
T.--M.\ Huang, Z.\ Jia, and W.--W.\ Lin.
\newblock On the convergence of Ritz pairs and refined Ritz vectors for quadratic eigenvalue problems.
\newblock {\em BIT Numer. Math.}, 53, pp. 941--958 (2013).

\bibitem{Kressner-Roman}
K.\ Kressner, and J.\ E.\ Roman.
\newblock Memory--efficient Arnoldi algorithms for linearizations of matrix polynomials in Chebyshev basis.
\newblock {\em Numer. Linear Algebr.}, 21, pp. 569--588 (2014).

\bibitem{TOAR}
D.\ Lu, Y.\ Su, and Z.\ Bai.
\newblock Stability analysis of the two--level orthogonal Arnoldi procedure.
\newblock {\em SIAM J. Matrix Anal. Appl.}, 37, pp. 195--214 (2016).

\bibitem{Q-Arnoldi}
K.\ Meerbergen.
\newblock The quadratic Arnoldi method for the solution of the quadratic eigenvalue problem.
\newblock {\em SIAM J. Matrix Anal. Appl.}, 30(4), pp. 1463--1482 (2008).


\bibitem{MP}
K.\ Meerbergen, and J.\ P\'erez.
\newblock Error analysis of the two--level orthogonal Arnoldi method for solving linearized polynomial eigenvalue problems.
\newblock In preparation (2017).

\bibitem{NN}
Y.\ Nakatsukasa, and V.\ Noferini,
\newblock \textit{On the stability of computing polynomial roots via confederate linearizations}.
\newblock {\em Math. Comp.}, 85 (301), pp. 2391--2425 (2016).

\bibitem{NP}
V.\ Noferini, and J.\ P\'erez.
\newblock Chebyshev rootfinding via computing eigenvalues of colleague matrices: when is it stable?
\newblock {Math. Comp.}, 86, pp. 1741--1767 (2016).

\bibitem{Pete}
G.\ W.\ Stewart, and J.\ --G.\ Sun.
\newblock {\em Matrix Perturbation Theory}, Academic Press, New York, 1990.

\bibitem{Pete2}
 G.\ W.\ Stewart.
Backward error bounds for approximate Krylov subspaces.
Technical report UMIACS TR--2001--32 CMSC TR--4247, University of Maryland, Institute for Advanced Computer Studies, Department of Computer Science (2001).

\bibitem{TOAR_origin}
Y.\ Su, J.\ Zhang, and Z.\ Bai.
\newblock A compact Arnoldi algorithm for Polynomial Eigenvalue Problems.
\newblock {\tt http://math.cts.nthu.edu.tw/Mathematics/RANMEP\%20Slides/Yangfeng\%20Su.pdf}

\bibitem{CORK}
R.\ Van Beeumen, K.\ Meerbergen, and W.\ Michiels.
\newblock Compact rational Krylov methods for nonlinear eigenvalue problems.
\newblock {\em SIAM J. Matrix Anal. Appl.}, 36, pp. 820--838 (2015).

\bibitem{VanDooren}
P.\ Van Dooren, and P.\ Dewilde.
\newblock The eigenstructure of an arbitrary polynomial matrix: computational aspects.
\newblock {\em Linear Algebra Appl.}, 50, pp. 545--579 (1983).

\bibitem{delay}
Y.\ Zhang, and Y.\ Su.
\newblock A memory--efficient model order reduction for time--delay systems.
\newblock {\em BIT Numer. Math.}, 53, pp. 1047--1073 (2013).


\end{thebibliography}
\end{document}